\documentclass[]{interact}

\usepackage{easybmat, tikz, float}
\usepackage{hyperref}
\usepackage[all]{xy}
\usepackage[active]{srcltx}
\usepackage{MnSymbol}
\usepackage{mathrsfs}

\usepackage{epstopdf}
\usepackage[caption=false]{subfig}

\usepackage[numbers,sort&compress]{natbib}
\bibpunct[, ]{[}{]}{,}{n}{,}{,}
\makeatletter
\def\NAT@def@citea{\def\@citea{\NAT@separator}}
\makeatother

\theoremstyle{plain}
\newtheorem{theorem}{Theorem}[section]
\newtheorem{lemma}[theorem]{Lemma}

\theoremstyle{definition}

\newtheorem{example}[theorem]{Example}

\theoremstyle{remark}
\newtheorem{remark}{Remark}

\begin{document}


\title{Schur decomposition of several matrices}

\author{
\name{Andrii Dmytryshyn\textsuperscript{a}\thanks{CONTACT Andrii Dmytryshyn. Email: andrii.dmytryshyn@oru.se. Address: School of Science and Technology, 
\"{O}rebro University, 
\"{O}rebro, SE-70182, 
Sweden.}}}

\maketitle

\begin{abstract}
Schur decompositions and the corresponding Schur forms of a single matrix, a pair of matrices, or a collection of matrices associated with the periodic eigenvalue problem are frequently used and studied. These forms are upper-triangular complex matrices or quasi-upper-triangular real
matrices that are equivalent to the original matrices via unitary or, respectively, orthogonal transformations.
In general, for theoretical and numerical purposes we often need to reduce, by admissible transformations, a collection of matrices to the Schur form. Unfortunately, such a reduction is not always possible.  
In this paper we describe all collections of complex (real) matrices that can be reduced to the Schur form by the corresponding unitary (orthogonal) transformations and explain how such a reduction can be done. 
We prove that this class consists of the collections of matrices associated with pseudoforest graphs. 
{ In other words,} we describe when the Schur form of a collection of matrices exists and how to find it.  
\end{abstract}


\section{Introduction} \label{intro}
%
%
The {\it Schur decomposition} \cite{Schu09} of a complex $n \times n$ matrix $A$ is the matrix decomposition $A = U T U^H$, where $T$ is an upper triangular matrix, $U$ is a unitary matrix, and $U^{H}$ denotes the conjugate transpose of $U$. If the matrix $A$ has real entries, then an analogous decomposition is $A = Q T Q^\top$, where $Q$ is an orthogonal matrix, $Q^\top$ denotes the transpose of $Q$, and $T$ is a quasi-upper-triangular matrix with $1 \times 1$ and $2 \times 2$ blocks on the main diagonal. The latter is also called the {\it Schur decomposition of a real matrix} $A$. The matrix $T$ (in both the real and complex cases) is called the {\it Schur form} of $A$. Schur decompositions became classical results in matrix analysis, see, e.g., \cite{GoVa96,HoJo85}. They are well studied and understood, in particular, due to their importance in applications. For the same reason, these results are extended to matrix pencils \cite{GoVa96}, 
as follows: 
$A_1 + \lambda A_2 = Q_1 T_1 Q^\top_2 + \lambda Q_1 T_2 Q^\top_2$, and to collections of matrices associated with the periodic eigenvalue problem  \cite{BoGV92, ByRh95, GrKK07b, SrDo93}: $(A_1,A_2, \dots , A_r) = (Q_2 T_1 Q^\top_1, Q_3 T_2 Q^\top_2, \dots , Q_1 T_r Q^\top_r)$. The latter decompositions are often called {\it generalized} and {\it periodic Schur decompositions}, respectively. From the applied point of view, it is often more convenient to see these decompositions as reductions of a matrix or a collection of matrices to (quasi-)upper-triangular forms, e.g., $Q^\top AQ = T$, see also Example  \ref{ex1}. 
In this paper, we describe which collections of matrices can be reduced to the Schur form,  i.e., to the (quasi-)upper-triangular form, by the corresponding transformations. In the other words, we describe when the Schur decomposition of a collection of matrices exists. 
The proof of the main result also provides a way to find the corresponding Schur form. 

To be able to describe any collection of matrices and the corresponding transformations easily, we associate them with a directed graph (this idea is borrowed from { representation theory, note also that directed graph are often called quivers in representation theory literature}). 
Recall that a {\it pseudotree} is an undirected connected graph that contains at most one cycle. A connected acyclic graph, i.e., a {\it tree}, is therefore a pseudotree. 
A {\it pseudoforest} is an undirected graph in which every connected component is a pseudotree. Note that a {\it forest}, i.e., a not-necessarily-connected acyclic graph, is a pseudoforest. 
%
%
We apply the terms ``tree'', ``pseudotree'' and ``pseudoforest'' to directed graphs, simply by ignoring the directions of the edges. 
%
A \emph{representation} of a directed graph is given by assigning to each vertex a vector space (over the field of either real or complex numbers) and to each direc ted edge a linear mapping. By choosing bases in these vector spaces, we can express the mappings { between these spaces by matrices.} Reselection of the bases reduces by corresponding equivalence  transformations the matrices of these linear mappings, 
for details see \cite{HoSe14,Serg88} as well as the following example. 

\begin{example} \label{ex1}
For the graphs in the table in Figure \ref{tabl}: ${\cal V}, {\cal V}_1, {\cal V}_2, \dots ,{\cal V}_n$ are vector spaces over the field of real numbers, $A, A_1,A_2,\dots ,A_n$ are matrices of the linear mappings, and $Q, Q_1, Q_2,\dots ,Q_n$ are orthogonal matrices that are used to change bases in the vector spaces ${\cal V}, {\cal V}_1,{\cal V}_2,\dots ,{\cal V}_n$, respectively. How the matrices $A, A_1,A_2, \dots ,A_n$ are transformed when we change the bases of ${\cal V}, {\cal V}_1,{\cal V}_2,\dots ,{\cal V}_n$ with matrices $Q, Q_1,Q_2,\dots ,Q_n$ is presented in the third column. 
\begin{figure}
\begin{table}[H]
\centering
\begin{tabular}[t]{ c c c  c }
&&&\\ \hline
&Graph & Representation of the graph& Transformation \\ \hline
(a)&\begin{tikzpicture}[->,shorten >=1pt,auto,node distance=3cm,
        thick,main 
        node/.style={circle,fill=black!100,draw,minimum size=0.2cm,inner sep=0pt]}]

    \node[main node] (1) {};
    \path[->]
    (1) edge [loop above, distance=1cm] node {} (1);
\end{tikzpicture} &
\begin{tikzpicture}[->,shorten >=1pt,auto,node distance=3cm,
        thick,main 
        node/.style={circle,draw,minimum size=0.7cm,inner sep=0pt]}]

    \node[main node] (1) {${\cal V}$};
    \path[->]
    (1) edge [loop above] node {$ A $} (1);
\end{tikzpicture} & 
$A \mapsto Q^\top A Q=:T$ \\ 
&&&\\ 
\hline
(b)&\begin{tikzpicture}[->,shorten >=1pt,auto,node distance=2cm,
        thick,main 
        node/.style={circle,fill=black!100,draw,minimum size=0.2cm,inner sep=0pt]}]

    \node[main node] (1) {};
    \node[main node] (2) [right of=1] {};

    \path[->]
    (1) edge [above] node 
{} 
(2);
\end{tikzpicture}&
\begin{tikzpicture}[->,shorten >=1pt,auto,node distance=2cm,
        thick,main 
        node/.style={circle,draw,minimum size=0.7cm,inner sep=0pt]}]

    \node[main node] (1) {${\cal V}_1$};
    \node[main node] (2) [right of=1] {${\cal V}_2$};

    \path[->]
    (1) edge [above] node 
{$A$} 
(2);
\end{tikzpicture}& 
$A \mapsto Q^\top_2 A Q_1=:T$
\\
&&&\\ 
\hline 
$\begin{matrix}
{\rm(c)}\\
\quad \\
\quad
\end{matrix}$
&\begin{tikzpicture}[->,shorten >=1pt,auto,node distance=2cm,
        thick,main 
        node/.style={circle,fill=black!100,draw,minimum size=0.2cm,inner sep=0pt]}]

    \node[main node] (1) {};
    \node[main node] (2) [right of=1] {};

    \path[->]
    (1) edge [bend left] node 
{} 
(2)
    (1) edge [bend right] node [below]
{} 
(2);
\end{tikzpicture} &
\begin{tikzpicture}[->,shorten >=1pt,auto,node distance=2cm,
        thick,main 
        node/.style={circle,draw,minimum size=0.7cm,inner sep=0pt]}]

    \node[main node] (1) {${\cal V}_1$};
    \node[main node] (2) [right of=1] 
    {
    ${\cal V}_2$ 
    };

    \path[->]
    (1) edge [bend left] node 
{$A_1$} 
(2)
    (1) edge [bend right] node [below]
{$A_2$} 
(2);
\end{tikzpicture}
& 
$
\begin{matrix}
A_1 \mapsto Q^\top_2 A_1 Q_1=:T_1\\
A_2 \mapsto Q^\top_2 A_2 Q_1=:T_2\\
\quad \\
\quad
\end{matrix}
$\\
\hline 
$\begin{matrix}
{\rm(d)}\\
\quad \\
\quad
\end{matrix}$
&\begin{tikzpicture}[->,shorten >=1pt,auto,node distance=2cm,
        thick,main 
        node/.style={circle,fill=black!100,draw,minimum size=0.2cm,inner sep=0pt]}]

    \node[main node] (1) {};
    \node[main node] (2) [right of=1] {};

    \path[->]
    (1) edge [bend left] node 
{} 
(2)
    (2) edge [bend left] node [below]
{} 
(1);
\end{tikzpicture} &
\begin{tikzpicture}[->,shorten >=1pt,auto,node distance=2cm,
        thick,main 
        node/.style={circle,draw,minimum size=0.7cm,inner sep=0pt]}]

    \node[main node] (1) {${\cal V}_1$};
    \node[main node] (2) [right of=1] 
    {
    ${\cal V}_2$ 
    };

    \path[->]
    (1) edge [bend left] node 
{$A_1$} 
(2)
    (2) edge [bend left] node [below]
{$A_2$} 
(1);
\end{tikzpicture}
& 
$
\begin{matrix}
A_1 \mapsto Q^\top_2 A_1 Q_1=:T_1\\
A_2 \mapsto Q^\top_1 A_2 Q_2=:T_2\\
\quad \\
\quad
\end{matrix}
$\\
\hline 
(e)&\begin{tikzpicture}[->,shorten >=1pt,auto,node distance=1cm,
        thick,main 
        node/.style={circle,fill=black!100,draw,minimum size=0.2cm,inner sep=0pt]}]
    \node[main node] (1)[label=right:$ \ \dots$] {};
    \node[main node] (3) [right of=1]  {};
    \node[main node] (2)  [left of=1] {};
    \node[main node] (5) [right of=3]  {};
\path[->]
(2) edge  node [] {} (1)
 (3) edge  node [] {} (5)
  (5) edge [bend right] node [] {} (2);
\end{tikzpicture} &
\tiny{
\begin{tikzpicture}[->,shorten >=1pt,auto,node distance=1.3cm,
        thick,main 
        node/.style={circle,draw,minimum size=0.58cm,inner sep=0pt]}]
    \node[main node] (1)[label=right:$\ \boldsymbol \dots$] {${\cal V}_2$};
    \node[main node] (2)  [left of=1] {${\cal V}_1$};
    \node[main node] (3) [right of=1]  {${\cal V}_{n-1}$};
    \node[main node] (5) [right of=3]  {${\cal V}_n$};
\path[->]
(2) edge  node [label=below: $A_1$] {} (1)
 (3) edge  node [label=below: $A_{n-1}$] {} (5)
  (5) edge [bend right] node [label=above: $A_n$] {} (2);
\end{tikzpicture}}
& 
$
\begin{matrix}
A_i \mapsto Q^\top_{i+1} A_i Q_i=:T_i,\\
i=1..n, \ Q_{n+1}:=Q_1\\
\quad 
\end{matrix}
$
\\
  \end{tabular}
  \label{table}
\end{table}
    \caption{{ The row (a) of this table shows how the problem of finding the Schur form of a single matrix can be associated with a graph with one vertex and one loop (a loop is a cycle of length one, i.e., with only one edge). The row (b) shows how the problem of reduction to an upper-triangular form of a single matrix using different orthogonal transformations from the left and right can be associated with a graph with two vertices and one edge. The row (c) presents how the problem of finding the Schur form of a pair of matrices (or equivalently the problem of finding the Schur form of a matrix pencil) can be associated with a graph with two vertices and two edges in-between them (so called Kronecker quiver). { The row (d) presents} how the problem of finding the Schur form of a pair of matrices, using different admissible transformations than in (c), can be associated with a graph with two vertices and two oppositely directed edges in-between these vertices (so called contragredient quiver). Finally, the row (e) presents how the problem of finding the periodic Schur form can be associated with a cyclic graph. Note also that all the graphs in  the table are pseudotrees and the graph in the row (b) is a tree (for examples of graphs that are not pseudotrees,  see, e.g., Figures \ref{any2c} and \ref{polpic}).}} 
    \label{tabl}
\end{figure}
\end{example}
In this paper, we show that a collection of real matrices of linear mappings associated with a pseudoforest can be reduced to a quasi-upper-triangular form by the corresponding orthogonal transformations, as well as that a collection of complex matrices of linear mappings associated with a pseudoforest can be reduced to upper-triangular forms by the corresponding unitary transformations. 
%
%
%
We also show that pseudoforests are the ``most complicated'' graphs whose representations are reduced to the Schur from. 
The latter means that for the graphs with two cycles a simultaneous \mbox{(quasi-)}upper-triangularization of the associated matrices is not possible in general.  




{ Note that, in practice, we} are typically dealing with particular cases of the result of this paper. Thus we are given a collection of matrices and their admissible transformations. Nevertheless, this information allows us to determine the associated { directed graph} immediately and without requiring any additional input data, { see, e.g., how the graph is constructed in Figure \ref{polpic}.}  
%
%
\section{Schur form or (quasi-)upper-triangularization of collections of matrices}

In this section, we present our main results in both real and complex cases: for a given collection of matrices we determine when the Schur decomposition exists 
and how to find it ({ the} procedure is { provided in} the proof of Theorem \ref{real}). 

\begin{theorem} \label{real}
Let $A_i, i = 1, \dots , n,$ be $r_{j(i)} \times r_{k(i)}$ real matrices of linear mappings associated with a pseudoforest. { Then there} exist orthogonal matrices $Q_1, \dots , Q_m,$ of compatible sizes such that    
\begin{equation*} 
\begin{aligned}
Q_{j(1)}^\top A_1Q_{k(1)}=T_1,\ \dots \ , \ Q_{j(n)}^\top A_nQ_{k(n)}=T_n, 
\end{aligned}
\end{equation*}
where $j(\cdot),k(\cdot): \{1, \dots ,n\} \to \{1, \dots ,m\}$ are mappings defined by the pseudoforest, $T_i, i = 2, \dots , n$ are $r_{j(i)} \times r_{k(i)}$ upper-triangular matrices, and $T_1$ is $r_{j(1)} \times r_{k(1)}$ quasi-upper-triangular with $1 \times 1$ and $2 \times 2$ blocks on the diagonal. 
\end{theorem}

\begin{remark}
\label{rem1}
\noindent { In the following we explain what being ``quasi-upper-triangular'' means for rectangular matrices. In the notation of Theorem \ref{real},} for each pseudotree: If $r_{j(i)} \neq r_{k(i)}$ then the matrices $T_i$ associated with the cycle are of the forms a) and b) in Figure~\ref{shapes} for the matrices associated with the arrows pointing at one direction (we can choose the direction with the largest number of arrows pointing at) and of the forms c) and d) in Figure \ref{shapes} for the matrices associated with the arrows pointing at the opposite direction. 
The remaining matrices $T_i$, associated with the trees, are of the forms a) and b) in Figure \ref{shapes} if the transformation matrix with the larger index, i.e., the one associated with the vertex further away from the cycle, (either $Q_{j(i)}$ or $Q_{k(i)}$), also has larger dimension and of the forms c) and d) in Figure~\ref{shapes} otherwise.    
\begin{figure}[]
    \centering
\begin{tikzpicture}
\draw (3,0) -- (3,2) -- (4,1) -- (4,0) -- (3,0);
\draw [fill=black] (3,2) -- (4,2) -- (4,1) -- (3,2);
\node [left, black] at (2.85,0.25) {a)};
\end{tikzpicture}
\qquad
\begin{tikzpicture}
\draw (3,0) -- (3,1) -- (4,1) -- (5,0) -- (3,0);
\draw [fill=black] (4,1) -- (5,1) -- (5,0) -- (4,1);
\node [left, black] at (2.85,0.25) {b)};
\end{tikzpicture}
\qquad
\begin{tikzpicture}
\draw (3,0) -- (3,1) -- (4,0) -- (3,0);
\draw [fill=black] (3,1) -- (3,2) -- (4,2) -- (4,0) -- (3,1);
\node [left, black] at (2.85,0.25) {c)};
\end{tikzpicture}
\qquad
\begin{tikzpicture}
\draw (3,0) -- (3,1) -- (4,0) -- (3,0);
\draw [fill=black] (3,1) -- (5,1) -- (5,0) -- (4,0) -- (3,1);
\node [left, black] at (2.85,0.25) {d)};
\end{tikzpicture}
    \caption{Rectangles represent matrices and the white parts of the rectangles represent zero entries of the matrices. Matrices represented by a) and c) have fewer columns than rows and by b) and d) have fewer rows than columns.}  
    \label{shapes}
\end{figure}
\end{remark}

\begin{proof}[Proof of Theorem \ref{real}]
Since every pseudotree is reduced independently, it is enough to prove the theorem for a pseudotree.   

Denote by $n'$ the length of the cycle. Note that $n'=1$ if a pseudotree has a loop, and $n'=2$ if a pseudotree contains the Kronecker or contragredient quiver, see, e.g., Example \ref{ex1}. 
By \cite{BoGV92, ByRh95, GrKK07b, Serg04, SrDo93} (see also Remark \ref{cancycle}) we can reduce the $n'$ matrices that form the cycle to the quasi-upper-triangular form: 
\begin{equation*} 
\begin{aligned}
Q_{j(1)}^\top A_1Q_{k(1)}=T_1,\ \dots \ , \ Q_{j(n')}^\top A_{n'}Q_{k(n')}=T_{n'}.  
\end{aligned}
\end{equation*}
The $n'$ transformation matrices $Q_{j(1)}, \dots , Q_{j(n')}, Q_{k(1)}, \dots , Q_{k(n')}$ (each written twice here) are now fixed. We use $Q_{j(1)}, Q_{k(1)}, \dots$ since we do not fix any directions of the edges; e.g., the matrix $A_1$ can be changed either as $Q_2^\top A_1 Q_1$, or as $Q_1^\top A_1Q_2$ depending on the direction of the edge. If $n>n'$ then there is at least one vertex { (among the $n'$ vertices of the cycle)} of degree at least 3. Assume that it is vertex $1$ (here we consider the general case but one can see Figure \ref{ptree1} as an example). 

\begin{figure}[H]
    \centering
\begin{tikzpicture}[->,shorten >=1pt,auto,node distance=3cm,
        thick,main 
        node/.style={circle,draw,minimum size=0.88cm,inner sep=0pt]}]

\node[main node] (1) {${\cal V}_1$};
\node[main node] (2) [below left of=1] {${\cal V}_2$};
\node[] (3) [above left of=2] {$\dots$};
\node[main node] (4) [above left of=1] {${\cal V}_{n'}$};
\node[main node] (5) [above right of=1] {${\cal V}_{n'+1}$};
\node[main node] (6) [right of=5] {${\cal V}_{n'+2}$};
\node[main node] (7) [below right of=5] {${\cal V}_{n'+3}$};
\node[main node] (8) [below right of=1] {${\cal V}_{n'+4}$};

\path[->]
      (1) edge [bend right] node [right] [label=above: $A_{n'}$] {} (4)
           edge [bend left] node [right] {} (2)
           edge [] node [right] [label=below: $A_{n'+1}$] {} (5)
      (2) edge [bend right,-] node [right] [label=below: $A_1$] {} (1)
          edge [bend left,-] node [right] {} (3)  
      (4) edge [bend right,-] node [right] [label=below: $A_{n'-1}$] {} (3)
          edge [bend left,-] node [right] {} (1)
      (3) edge [bend right] node [right] [label=above: $A_2$] {} (2)
          edge [bend left] node [right] {} (4)
      (5) edge [] node [right] [label=above: $A_{n'+2}$] {} (6) 
      (7) edge [] node [right] [label=above: $A_{n'+3}$] {} (5)
      (8) edge [] node [right] [label=above: $A_{n'+4}$] {} (1);       
\end{tikzpicture}
    \caption{A part of a pseudotree with a cycle of length $n'$. Two tree graphs are connected to the cycle at vertex 1: The first one has 3 vertices, indexed $n'+1,n'+2,n'+3$, and the second one has 1 vertex, indexed $n'+4$.} 
    \label{ptree1}
\end{figure}

Now consider the edge that is not involved in the cycle but starts a tree connected to the cycle, let it correspond to $A_{n'+1}$. Depending on the direction of this edge we have the following two cases:  
\begin{itemize}
\item The edge is directed ``from the cycle'' (as in Figure \ref{ptree1}); then $A_{n'+1}$ is changed as $Q_{n'+1}^\top A_{n'+1}Q_1$, where $Q_1$ is fixed. In this case, we choose $Q_{n'+1}$ to be equal to ``$Q$'' from the $QR$ decomposition of $A_{n'+1}Q_1$,  i.e., $A_{n'+1}Q_1= Q_{n'+1}R_{n'+1}$ and thus $Q_{n'+1}^\top A_{n'+1}Q_1= R_{n'+1}$. Clearly, $T_{n'+1}:= R_{n'+1}$.

\item The edge is directed ``to the cycle''; then $A_{n'+1}$ is changed as $Q_1^\top A_{n'+1}Q_{n'+1}$, where $Q_1$ is fixed.  In this case, we choose $Q_{n'+1}$ to be equal to ``$Q$'' from the $RQ^\top$ decomposition of $Q_1^\top A_{n'+1}$,  i.e., $Q_1^\top A_{n'+1}= R_{n'+1}Q^\top_{n'+1}$ and thus $Q_1^\top A_{n'+1}Q_{n'+1}= R_{n'+1}$. Again, $T_{n'+1}:= R_{n'+1}$.
\end{itemize}
Both of the above cases use $Q_{n'+1}$ for the reduction. This results { in} fixing $Q_{n'+1}$. Note that if the dimension of ${\cal V}_{n'+1}$ is smaller than the dimension of ${\cal V}_{1}$, then $A_{n'+1}$ is reduced to one of the forms c) or d) in Figure~\ref{shapes} (sometimes called upper-trapezoidal); if the dimension of ${\cal V}_{n'+1}$ is larger than the dimension of ${\cal V}_{1}$, then $A_{n'+1}$ is reduced to one of the forms a) or b) in Figure~\ref{shapes}. 

Now, this procedure must be done for all the remaining edges sharing the vertex $n'+1$, if there are any. We repeat this procedure until we reach the end of the tree (leaves). 

If the degree of vertex 1 is greater than three, then there { is} more than one tree connected to the cycle at this vertex, e.g., in Figure \ref{ptree1} the degree of vertex 1 is four. We repeat the reduction above for each such a tree, then move to the next vertex of the cycle and reduce all the trees connected to the cycle there, etc.   
\end{proof}


{ We say that a graph has two cycles if  there is an edge in the first cycle that does not belong to the second cycle and there is an edge in the second cycle that does not belong to  the  first cycle.} What remains to show is that if { any of the connected components of a} graph has two cycles then a simultaneous quasi-upper-triangularization of the corresponding matrices is not possible, in general. In Lemma \ref{pair} we show that collections of matrices associated with the simplest graph containing two cycles, i.e., the graph with only one vertex and two loops, can not be quasi-upper-triangularized. { In Theorem \ref{l2c} we show that the example given in Lemma \ref{pair} is generalizable to any connected graph with two cycles. }
\begin{lemma}\label{pair}
Let $A$ be a non-quasi-diagonalizable real matrix, 
( i.e., it cannot be reduced to a diagonal form with $1 \times 1$ and $2 \times 2$ blocks on the diagonal). { Then there} is no orthogonal matrix $Q$ such that $Q^\top A Q$ and $Q^\top A^\top Q$ are both quasi-upper-triangular.    
\end{lemma}
\begin{proof}
Since $A$ is a non-quasi-diagonalizable real square matrix, it is enough to notice that $(Q^\top AQ)^\top = Q^\top A^\top Q$. Therefore if $Q^\top A Q$ is quasi-upper-triangular (and not quasi-diagonal) then $Q^\top A^\top Q$ must be quasi-lower-triangular. The associated graph is presented in Figure \ref{2loops}. 
\end{proof}
\begin{figure}[H]
    \centering
\begin{tikzpicture}[->,shorten >=1pt,auto,node distance=3cm,
        thick,main 
        node/.style={circle,draw,minimum size=0.7cm,inner sep=0pt]}]

    \node[main node] (1) {${\cal V}$};
    \path
    (1) edge [loop right, distance=1cm] node {$A^\top$} (1)
         edge [loop left, distance=1cm] node {$A$} (1);
\end{tikzpicture}
    \caption{A representation of a graph with only one vertex and two loops. In general, by orthogonal changes of the basis of the vector space ${\cal V}$, we can not get both of the matrices $A$ and $A^\top$ being quasi-upper-triangular.} 
    \label{2loops}
\end{figure} 
{ 
\begin{theorem}\label{l2c}
Let $G$ be a graph that has a connected component with two cycles. Then there is a  collection  of real matrices associated with $G$ such that a simultaneous quasi-upper-triangularization (as the one described in Theorem \ref{real}), by admissible transformations (defined by $G$), of this collection of matrices is not possible.
\end{theorem}
\begin{proof}
Consider any connected graph with two cycles. We show that  the following collection of matrices, associated with such graph, cannot be reduced to a quasi-upper-triangular form: a non-quasi-diagonalizable real matrix $A$ is associated with one edge in one of the cycles, the matrix $A^T$ is associated with an edge in the second cycle (any edge that does not belong to the first  cycle can be chosen), and the identity matrices are associated with all the other edges, see, e.g., the part of the graph in Figure \ref{any2c}. 
\begin{figure}[H]
    \centering
\tiny{
\begin{tikzpicture}[->,shorten >=1pt,auto,node distance=2cm,
        thick,main 
        node/.style={circle,draw,minimum size=0.73cm,inner sep=0pt]}]

\node[main node] (1) {${\cal V}_1$};
\node[main node] (2) [below left of=1] {${\cal V}_2$};
\node[] (3) [above left of=2] {$\dots$};
\node[main node] (4) [above left of=1] {${\cal V}_{n'}$};
\node[] (5) [right of=1] {$\dots$};
\node[main node] (6) [right of=5] {${\cal V}_{n''+1}$};
\node[main node] (7) [below right of=6] {${\cal V}_{n'''}$};
\node[main node] (8) [above right of=6] {${\cal V}_{n''+2}$};
\node[] (9) [below right of=8] {$\dots$};

\path[->]
      (1) edge [bend right] node [right] [label=above: $I$] {} (4)
           edge [bend left] node [right] {} (2)
           edge [] node [right] [label=above: $I$] {} (5)
      (2) edge [bend right,-] node [right] [label=below: $A$] {} (1)
          edge [bend left,-] node [right] {} (3)  
      (4) edge [bend right,-] node [left] [label=above: $I$] {} (3)
          edge [bend left,-] node [right] {} (1)
      (3) edge [bend right] node [left] [label=below: $I$] {} (2)
          edge [bend left] node [right] {} (4)
      (5) edge [] node [right] [label=above: $I$] {} (6) 
      (6) edge [bend right] node [left] [label=below: $I$] {} (7)
           edge [bend left] node [left] [label=above: $A^T$] {} (8)
      (7) edge [bend right] node [right] [label=below: $I$] {} (9)
      (8) edge [bend left] node [right] [label=above: $I$] {} (9);       
\end{tikzpicture}}
    \caption{A part of a connected graph with two cycles. $A$ and $I$ are square matrices of  the same size. $I$ is the identity matrix and $A$ is non-quasi-diagonalizable. One edge that belongs only to  the first cycle is associated with $A$ and one edge that belongs only to the second cycle is associated with $A^T$. All the other edges of the graph are associated with the identity matrices $I$.} 
    \label{any2c}
\end{figure}
Assume that such a collection can be reduced to a quasi-upper-triangular form, i.e., matrices $A_i, i=1,\ldots , n'''$ can be reduced to a quasi-upper-triangular $T_1$ and  upper-triangular $T_i, i=2,\ldots , n'''$.  
Then consider two paths in the graph that start at the same node and each of them includes a different cycle { (we pass through the nodes of each cycle only once)}. 
We use the graph in Figure \ref{any2c} for illustration but our arguments are general. By our assumption, the   matrices  $A, I, \ldots , I$ associated with  all the edges included in the first path (the first path starts at ${\cal V}_1$, then goes to ${\cal V}_2$, \dots, then  to ${\cal V}_{n'}$ and ends at ${\cal V}_1$) can be reduced as follows:
$
Q_{2}^\top AQ_{1}=T_1, \ Q_{2}^\top IQ_{3}=T_2,\ \dots \ , \ Q_{n'}^\top IQ_{n'-1}=T_{n'-1}, \ Q_{n'}^\top I Q_{1}=T_{n'}  
$ 
and the matrices associated with  all the edges in the second path (the second path starts at ${\cal V}_1$, then goes to ${\cal V}_{n''+1}$, to ${\cal V}_{n''+2}$, \dots, then  to ${\cal V}_{n'''}$, to ${\cal V}_{n''+1}$, and all the way back to ${\cal V}_1$) can be reduced as follows: 
$
Q_{n'+1}^\top I Q_{1}=T_{n'+1},\ \dots \ , \ Q_{n''+1}^\top IQ_{n''}=T_{n''}, \ Q_{n''+2}^\top A^TQ_{n''+1}=T_{n''+1}, \ Q_{n''+3}^\top I Q_{n''+2}=T_{n''+2},\ \dots \ , \ Q_{n'''-1}^\top IQ_{n'''}=T_{n'''-1}, \ Q_{n'''}^\top I Q_{n''+1}=T_{n'''}.  
$
Then by multiplying the matrices $T_i$ or their inverses in the order that follows the path  (the inverse is taken if the arrow is in the opposite direction to the direction that we are following) we obtain, for the first path:  
$$T_{n'}^{-1} T_{n'-1} \dots T_2^{-1} T_1   = Q_{1}^{\top} I Q_{n'}Q_{n'}^\top IQ_{n'-1} \dots Q_{3}^{\top} IQ_{2} Q_{2}^\top AQ_{1}  = Q_{1}^{\top}AQ_{1}$$
and,  for the second path: 
\begin{equation*}
\begin{aligned}
&T_{n'+1}^{-1} \dots T_{n''}^{-1}T_{n'''}^{-1} T_{n'''-1}^{-1} \dots T_{n''+2}T_{n''+1}T_{n''}  \dots  T_{n'+1}  =  Q_{1}^\top I Q_{n'+1} \cdot \\
&\ldots \cdot Q_{n''}^\top IQ_{n''+1} Q_{n''+1}^{\top} I Q_{n'''} Q_{n'''}^{\top} IQ_{n'''-1} \dots Q_{n''+3}^\top I Q_{n''+2}Q_{n''+2}^\top A^TQ_{n''+1}Q_{n''+1}^\top IQ_{n''}  \cdot \\ 
&\ldots \cdot Q_{n'+1}^\top I Q_{1} =  Q_{1}^\top A^T Q_{1}  \\
\end{aligned}
\end{equation*}
Since the product of $T_i$ from the first path is quasi-upper-triangular and equal to $Q_{1}^{\top}AQ_{1}$, and the product from the second path is upper-triangular and equal to $Q_{1}^\top A^T Q_{1}$, we have a contradiction by Lemma \ref{pair}.
Note that when the matrices in our collection are $2 \times 2$ and the matrix $T_1$ is $2 \times 2$ then $Q_{1}^{\top}AQ_{1}$ is a full $2 \times 2$ matrix but $Q_{1}^\top A^T Q_{1}$ is supposed to be upper-triangular (as  a  product of upper-triangular matrices) thus we still have a contradiction. 
\end{proof}
}
%

In the following theorem, { we present a result on a reduction of a collection of complex matrices to upper-triangular forms (see Remark \ref{rem1} and Figure \ref{shapes} for the explanation on what being ``upper-triangular'' means for rectangular matrices), which is analogous to the result for real matrices presented in Theorem \ref{real}.}
\begin{theorem}\label{compl}
Let $A_i, i = 1, \dots , n,$ be $r_{j(i)} \times r_{k(i)}$ complex matrices of linear mappings associated with a pseudoforest. { Then there} exist unitary matrices $U_1, \dots , U_m,$ of compatible sizes such that    
\begin{equation*} 
\begin{aligned}
U_{j(1)}^{H}A_1U_{k(1)}=T_1,\ \ \ \dots \ , \ U_{j(n)}^{H}A_nU_{k(n)}=T_n, 
\end{aligned}
\end{equation*}
where $j(\cdot),k(\cdot): \{1, \dots ,n\} \to \{1, \dots ,m\}$ are mappings defined by the pseudoforest, and $T_i, i = 1, \dots , n,$ are $r_{j(i)} \times r_{k(i)}$ upper-triangular matrices. 


\end{theorem}
\begin{proof}
The proof is analogous to the proof of Theorem \ref{real}.
\end{proof}

Note also that the counterexample in Lemma \ref{pair} works also for complex matrices under unitary transformations. In this case, we have to pick a non-diagonalizable square matrix $A$.  

{ 
\begin{theorem}\label{any2cc}
Let $G$ be a graph that has a connected component with two cycles. Then there is a  collection of complex matrices associated with $G$ such that a simultaneous upper-triangularization (as the one described in Theorem \ref{compl}), by admissible transformations (defined by $G$), of these matrices is not possible.
\end{theorem}
\begin{proof}
The proof is analogous to the proof of Theorem \ref{l2c}.
\end{proof}
}

\begin{remark}[Periodic Schur form or reduction for the cycles] \label{cancycle}
In the references \cite{BoGV92,  ByRh95,  GrKK07b,  SrDo93}, the authors mainly restrict themselves to the cases of square matrices. { Therefore, we give this short remark on the reduction to the Schur forms} of the rectangular matrices associated with the cycles.  Our explanation is based on a ``stronger'' result of \cite{Serg04}. In \cite{Serg04}, a Kronecker-like canonical form under the nonsingular transformations of the matrices associated with cycles is derived. By { applying} the $RQ$ decomposition of each of the transformation matrices and then multiplying the canonical matrices with $R$, we obtain the (quasi-)upper-triangular matrices. The remaining parts of the transformation matrices,  i.e., matrices $Q$, are the new unitary (orthogonal) transformation matrices. { For the details on this method see \cite[Theorem 2]{DeDD20a} where such a result is proved for symmetric matrix pencils.} We also refer to Section 5 of \cite{KaKK11} for the numerically stable procedure of reducing the case with the rectangular matrices to the case of square matrices for the periodic Schur form. 
\end{remark}

\section{Applications and future work} 
\label{appfw}
In this section we highlight some future work and possible applications of the results of this paper. In particular, our results show immediately that the problem of triangularization of complex matrix polynomials cannot be solved using unitary transformations; and give a possible simple form for cross-correlation matrices in statistical signal processing. They can also be used for investigating uniqueness of the solutions and solving some systems of Sylvester matrix equations. 

\noindent{\bf Triangularization of matrix polynomials}.
Let 
\begin{equation} 
\label{matpol}
P(\lambda) = \lambda^{d}A_{d} + \dots +  \lambda A_1 + A_0, 
\quad \ A_i \in \mathbb C^{m \times n},  \text{  and } i=0, \dots, d 
\end{equation}
be a matrix polynomial. We would like to know whether it is possible to triangularize $P(\lambda)$ by multiplying it from the left and right with unitary matrices $U_2$ and $U_1$, i.e., if we can find $U_2$ and $U_1$ such that $U_2 P(\lambda) U_1 =  \lambda^{d}U_2A_{d}U_1 + \dots +  \lambda U_2A_1U_1 + U_2A_0U_1$ is upper-triangular. 

Note that we already know that the answer is positive for a matrix polynomial of degree one, i.e., for a matrix pencil $\lambda A_1 + A_0$. Multiplying $\lambda A_1 + A_0$ from the left and right with unitary matrices $U_2$ and $U_1$, we obtain 
$U_2(\lambda A_1 + A_0)U_1 = \lambda U_2A_1U_1 + U_2A_0U_1 = \lambda T_1 + T_0,$
where both $T_1$ and $T_0$ are upper-triangular. 

Nevertheless, by Theorem \ref{any2cc}, the answer to our question is negative for $d \geq 2$. Namely, we have the same transformation applied to the rows of all the matrix coefficients of $P(\lambda)$ and the same transformation applied to the columns of all the matrix coefficients of $P(\lambda)$. Therefore the associated quiver is not a pseudotree, see Figure \ref{polpic}. 

\begin{figure}[H]
    \centering
\begin{tikzpicture}[->,shorten >=1pt,auto,node distance=4cm,
        thick,main 
        node/.style={circle,draw,minimum size=0.7cm,inner sep=0pt]}]

    \node[main node] (1) {${\cal V}_1$};
    \node[main node] (2) [right of=1] 
    {
    ${\cal V}_2$ 
    };

    \path[->]
    (1) edge [bend left] node 
{$A_0$} 
(2)
    (1) edge node {$A_1$} 
(2)
    (1) edge [bend right] node [label=above: $\boldsymbol \vdots$] [below]
{$A_d$} 
 (2);
\end{tikzpicture}
    \caption{A representation of a graph with two vertices and $d$ edges, $d \ge 2$. In general, by unitary changes of the bases of the complex vector spaces ${\cal V}_1$ and ${\cal V}_2$, we can not get all the matrices $A_i, i=0, \dots , d$ being upper-triangular.} 
    \label{polpic}
\end{figure} 
\noindent In \cite{ADHM21, TaTZ13} it is discussed how to triangularize polynomials using unimodular transformations. 

\noindent{\bf Reduction of correlation matrices in joint independent subspace analysis}.
Consider a collection of cross-correlation matrices $\{ S^{[k,l]}, k,l = 1, \dots, n \}$ coming form the model for statistical signal processing, called Joint Independent Subspace Analysis (JISA), see \cite{LaJu18} for more details on the model and construction of such matrices. 
JISA model allows us to transform these cross-correlation matrices as follows: $S^{[k,l]} \mapsto Z^{[k]}S^{[k,l]}Z^{[l]H}$, where the matrices $Z^{[i]}, i = 1, \dots, n$ are coming from coupled change of bases. Using our graph interpretation we note that such a base-change can be associated with a complete (fully-connected) graph on $n$ vertices. { These change of bases can be used,  e.g., for an investigation whether the matrices $\{ S^{[k,l]} \}$ are reducible or irreducible, for more details see \cite{LaJu18} and for the investigation of general (not only cross-correlation) double-indexed matrices see \cite{LaJS19}.}


%
\noindent{\bf Sylvester matrix equations}. 
Sylvester matrix equations and their systems come in all shapes and sizes, see, e.g., \cite{DIPR19,DFKS17,DmKa15}.  In \cite{DIPR19}, a criterion for the uniqueness { of the solution} of a square system of Sylvester matrix equations is presented. The result of this paper may be helpful for developing such a criterion in the case where the involved matrices are rectangular. Moreover, Schur form of a single matrix is a key ingredient of Bartels-Stewart algorithm for solving small-to-medium size Sylvester matrix equations $AX-XB =C$ and Schur form of the matrices associated with pseudoforests may be used for the Bartels-Stewart-type algorithm for solving systems of Sylvester matrix equations, especially if the involved matrices are rectangular. 
\section*{Acknowledgements}
The author is grateful to Vladimir Sergeichuk and Daniel Kressner for the useful discussions on this paper. { The author also thanks the anonymous referees for the helpful remarks and suggestions. 

The work of the author has been supported by the Swedish Research Council (VR) under grant 2021-05393.}

{\small 
\bibliographystyle{plain}
\bibliography{references}
}
\end{document}